\documentclass{amsart}%
\usepackage{amsmath}
\usepackage[tcidvi]{graphicx}
\usepackage{amsfonts}
\usepackage{amssymb}%
\providecommand{\U}[1]{\protect\rule{.1in}{.1in}}
\newtheorem{theorem}{Theorem}
\theoremstyle{plain}

\newtheorem{corollary}{Corollary}

\newtheorem{definition}{Definition}

\newtheorem{lemma}{Lemma}

\newtheorem{proposition}{Proposition}
\newtheorem{remark}{Remark}

\numberwithin{equation}{section}
\begin{document}
\title[ ]{A note on Abel's partial summation formula}
\author{Constantin P. Niculescu}
\address{The Academy of Romanian Scientists, Splaiul Independentei No. 54, Bucharest,
RO-050094, Romania.}
\curraddr{Department of Mathematics, University of Craiova, Craiova 200585, Romania}
\email{cpniculescu@gmail.com}
\author{Marius Marinel St\u{a}nescu}
\address{Department of Applied Mathematics, University of Craiova, Craiova 200585, Romania}
\email{mamas1967@gmail.com}
\thanks{}
\date{}
\subjclass{26A51; 26B25, 26D15, 46B40, 46B42, 47A35}
\dedicatory{ }\keywords{Abel's partial summation formula, Jensen-Steffensen inequality, convergence in
density, convex function, submajorization}

\begin{abstract}
Several applications of Abel's partial summation formula to the convergence of
series of positive vectors are presented. For example, when the norm of the
ambient ordered Banach space is associated to a strong order unit, it is shown
that the convergence of the series $\sum x_{n}$ implies the convergence in
density of the sequence $(nx_{n})_{n}$ to 0. This is done by extending the
Koopman-von Neumann characterization of convergence in density. Also included
is a new proof of the Jensen-Steffensen inequality based on Abel's partial
summation formula and a trace analogue of Tomi\'{c}-Weyl inequality of submajorization.

\end{abstract}
\maketitle

\section{Introduction}

Abel's partial summation formula (also known as Abel's transformation)
asserts\ that every pair of families $(a_{k})_{k=1}^{n}$ and $(b_{k}%
)_{k=1}^{n}$ of complex numbers verifies the identity%
\begin{equation}
\sum_{k=1}^{n}a_{k}b_{k}=\sum_{k=1}^{n-1}\left[  (a_{k}-a_{k+1})\left(
\sum_{j=1}^{k}b_{j}\right)  \right]  +a_{n}\left(  \sum_{j=1}^{n}b_{j}\right)
. \tag{$Ab^\uparrow$}\label{Aup}%
\end{equation}

This identity, that appears in the proof of Theorem III in \cite{Ab1826}, is
instrumental in deriving a number of important results such as the
Abel-Dirichlet criterion of convergence for signed series, the Abel theorem on
power series, the Abel summation method (see \cite{CN2014}, \cite{Str}),
Kronecker's lemma about the relationship between convergence of infinite sums
and convergence of sequences (see \cite{Shy}, Lemma IV.3.2, p. 390),
algorithms for establishing identities involving harmonic numbers and
derangement numbers \cite{CHJ2011}, the variational characterization of the
level sets corresponding to majorization in $\mathbb{R}^{N}$ \cite{Zhu2004},
Mertens' proof of his theorem on the sum of the reciprocals of the primes
\cite{V2005} etc.

Abel used his formula $($\ref{Aup}$)$ through an immediate consequence of it
(known as Abel's inequality): if $a_{1}\geq a_{2}\geq\cdots\geq a_{n}\geq0$
and $b_{1},b_{2},...,b_{n}\in\mathbb{C},$ then%
\[
\left\vert \sum_{k=1}^{n}a_{k}b_{k}\right\vert \leq a_{1}\max_{1\leq m\leq
n}\left\vert \sum_{k=1}^{m}b_{k}\right\vert .
\]

Many other striking applications of this inequality may be found in the books
of Pe\v{c}ari\'{c}, Proschan and Tong \cite{PPT} and Steele \cite{Steele}.

\begin{remark}
Abel's formula \emph{(\ref{Aup})} has the following\ backwards companion:%
\begin{equation}
\sum_{k=1}^{n}a_{k}b_{k}=\sum_{k=1}^{n-1}\left[  (a_{k+1}-a_{k})\left(
\sum_{j=k+1}^{n}b_{j}\right)  \right]  +a_{1}\left(  \sum_{j=1}^{n}%
b_{j}\right)  . \tag{$Ab^\downarrow$}\label{Adown}%
\end{equation}

A way to bring together the formulas \emph{(\ref{Aup})}\ and
\emph{(\ref{Adown})} is as follows:
\begin{align}
\sum_{k=1}^{n}a_{k}b_{k}  &  =\sum_{j=1}^{k-1}\left[  \left(  a_{j}%
-a_{j+1}\right)  \left(  \sum_{i=1}^{j}b_{i}\right)  \right]  +a_{k}\left(
\sum_{i=1}^{k}b_{i}\right) \tag{$Ab$}\label{Ab}\\
&  +a_{k+1}\left(  \sum_{i=k+1}^{n}b_{i}\right)  +\sum_{j=k+2}^{n}\left[
(a_{j}-a_{j-1})\left(  \sum_{i=j}^{n}b_{i}\right)  \right]  ,\nonumber
\end{align}
for any index $k.$
\end{remark}

It is worth noticing that the formulas $($\ref{Aup}$)$ and (\ref{Adown}) (as
well as \ref{Ab}) extend verbatim to the context of (not necessarily
commutative) bilinear maps
\[
\Phi:E\times F\rightarrow G,
\]
where\ $E,$ $F$ and $G$ are three vector spaces over the same base field
$\mathbb{K}$. For example, the following identities hold true for all families
$(x_{k})_{k=1}^{n}$ and $(y_{k})_{k=1}^{n}$ of elements belonging respectively
to $E$ and $F$:
\begin{align}
\sum_{k=1}^{n}\Phi\left(  x_{k},y_{k}\right)   &  =\sum_{k=1}^{n-1}\Phi\left(
x_{k}-x_{k+1},\sum_{j=1}^{k}y_{j}\right)  +\Phi\left(  x_{n},\sum_{j=1}%
^{n}y_{j}\right) \tag{$\Phi A1$}\label{B1}\\
&  =\sum_{k=1}^{n-1}\Phi\left(  \sum_{j=1}^{k}x_{j},y_{k}-y_{k+1}\right)
+\Phi\left(  \sum_{j=1}^{n}x_{j},y_{n}\right)  \tag{$\Phi A2$}\label{B2}%
\end{align}
and%
\begin{align}
\sum_{k=1}^{n}\Phi\left(  x_{k},y_{k}\right)   &  =\sum_{k=2}^{n}\Phi\left(
x_{k}-x_{k-1},\sum_{j=k}^{n}y_{j}\right)  +\Phi\left(  x_{1},\sum_{j=1}%
^{n}y_{j}\right) \tag{$\Phi A3$}\label{B3}\\
&  =\sum_{k=2}^{n}\Phi\left(  \sum_{j=k}^{n}x_{j},y_{k}-y_{k-1}\right)
+\Phi\left(  \sum_{j=1}^{n}x_{j},y_{1}\right)  , \tag{$\Phi A4$}\label{B4}%
\end{align}
Moreover, these identities also work (with obvious changes) when the summation
range is from $m$ to $n$ whenever $1\leq m\leq n$; this represents the special
case where $x_{1}=\cdots=x_{m-1}=0$ and $y_{1}=\cdots=y_{m-1}=0.$

The aim of this paper is to illustrate the formulas $(\Phi A1)$-$(\Phi A4)$ in
the context of ordered Banach spaces. For the convenience of the reader some
very basic facts concerning these spaces are recalled in the next section.
Then in Section 3 we present applications to the convergence of series in
ordered Banach spaces. Section 4 is devoted to a new short proof of the
Jensen-Steffensen inequality based on Abel's partial summation formula and to
an extension of this inequality to the framework of Banach lattices. Finally,
in Section 5 we prove a trace analogue of the Tomi\'{c}-Weyl inequality of submajorization.

\section{Preliminaries on ordered Banach spaces}

An ordered vector space is any real vector space $E$ endowed with a convex
cone $E_{+}$ (the cone of positive elements) such that
\[
E_{+}\cap(-E_{+})=\left\{  0\right\}  \text{ and }E=E_{+}-E_{+}.
\]
If $E$ is in the same time a Banach space, we call $E$ ordered Banach space
when the following compatibility condition between the two structures is
fulfilled:%
\[
0\leq x\leq y\text{ implies }\left\Vert x\right\Vert \leq\left\Vert
y\right\Vert .
\]
The usual real Banach spaces like $\mathbb{R}^{N}$ (the Euclidean
$N$-dimensional space)$,$ $C\left(  K\right)  $ (= the space of all continuous
functions defined on a compact Hausdorff space $K),$ the Lebesgue spaces
$L^{p}\left(  \mathbb{R}^{N}\right)  $ (for $1\leq p\leq\infty)$, as well as
their infinite dimensional discrete analogues $c$ and $\ell^{p})$ are endowed
with order relations that behave much better. Indeed, they are all Banach
lattices, that is, vector lattices (meaning the existence of $\max\left\{
x,y\right\}  $ and $\min\left\{  x,y\right\}  $ for every pair of elements)
plus the compatibility condition%
\[
\left\vert x\right\vert \leq\left\vert y\right\vert \text{ implies }\left\Vert
x\right\Vert \leq\left\Vert y\right\Vert ;
\]
here the modulus of an element $z$ is defined as $\left\vert z\right\vert
=\max\left\{  -z,z\right\}  .$

The order relation in a function space is usually the pointwise one defined by%
\[
f\leq g\text{ if and only if }f(t)\leq g(t)\text{ for all }t;
\]
this remark includes the case of $\mathbb{R}^{N},$ whose ordering is defined
by coordinates.

A bounded linear operator $T\in L(E,F)$ acting on ordered Banach spaces is
called \emph{positive} if it maps positive elements into positive elements.
Typical examples are the integration operators.

In the realm of Hilbert spaces $H$ one encounters a rather different concept
of positivity. Precisely, the Banach space $\mathcal{A}(H),$ of all bounded
self-adjoint linear operators $A:H\rightarrow H,$ becomes an ordered vector
space when endowed with the positive cone%
\[
\mathcal{A}(H)_{+}=\left\{  A\in\mathcal{A}(H):\text{ }\langle Ax,x\rangle
\geq0\text{ for all }x\in H\right\}  .
\]
Though this ordering does not make $\mathcal{A}(H)$ a Banach lattice, it has
many nice features exploited by the spectral theory of these operators. In
particular, $\mathcal{A}(H)$ is an ordered Banach space such that
\[
-A\leq B\leq A\text{ implies }\left\Vert B\right\Vert \leq\left\Vert
A\right\Vert
\]
and every order bounded increasing sequence of operators has a least upper
bound. Moreover, since%
\[
\left\Vert A\right\Vert =\sup_{\left\Vert x\right\Vert =1}\left\vert \langle
Ax,x\rangle\right\vert ,
\]
we have $\left\Vert A\right\Vert \leq M$ if and only if $-M\cdot I\leq A\leq
M\cdot I,$ where $I$ is the identity of $H.$ See Simon \cite{Sim}.

A nice account on the basic theory of Banach lattices and positive operators
may be found in the classical book of Schaefer \cite{S1974}, while the general
theory of ordered Banach spaces is made available by the books of Lacey
\cite{La1974} and Schaefer \cite{S1971}

In the next section we shall be interested in the following special class of
bilinear maps acting on ordered Banach spaces.

\begin{definition}
Suppose that $E,$ $F$ and $G$ are ordered Banach spaces. A bilinear map
$\Phi:E\times F\rightarrow G$ is called positive if%
\[
x\geq0\text{ and }y\geq0\text{ imply }\Phi(x,y)\geq0.
\]

\end{definition}

Notice that a positive bilinear map verifies the following propriety of
monotonicity:%
\[
0\leq x_{1}\leq x_{2}\text{ and 0}\leq y_{1}\leq y_{2}\text{ imply }%
\Phi\left(  x_{1},y_{1}\right)  \leq\Phi\left(  x_{2},y_{2}\right)  .
\]
Indeed, $\Phi\left(  x_{2},y_{2}\right)  -\Phi\left(  x_{1},y_{1}\right)
=\Phi(x_{2}-x_{1},y_{2})+\Phi(x_{1},y_{2}-y_{1})\geq0.$

Using formula (\ref{B2}) and the property of monotonicity \ one can prove
easily the following extension of Abel's inequality:

\begin{proposition}
Suppose that $\Phi:E\times F\rightarrow G$ is a positive bilinear map. If
$m\leq\sum_{i=1}^{k}x_{i}\leq M$ in $E$ $($for $k=1,...,n)$ and $y_{1}\geq
y_{2}\geq\cdots\geq y_{n}\geq0$ in $F,$ then from formula (\ref{B2}) it
follows that%
\[
\Phi(m,y_{1})\leq\sum_{k=1}^{n}\Phi(x_{k},y_{k})\leq\Phi(M,y_{1}).
\]

\end{proposition}

Notice also that a positive bilinear map acting on ordered Banach spaces is
always bounded, this meaning the existence of a positive constant $C$ such
that%
\[
\left\Vert \Phi\left(  x,y\right)  \right\Vert \leq C\left\Vert x\right\Vert
\left\Vert y\right\Vert \text{ for all }\left(  x,y\right)  \in E\times F.
\]
The proof follows easily by adapting the argument of Theorem 5.5 (ii) in
\cite{S1971}, p. 228. The smallest constant $C$ for which the above inequality
holds for all $\left(  x,y\right)  \in E\times F$ is called the norm of $\Phi$
and is denoted $\left\Vert \Phi\right\Vert .$

Examples of positive bilinear maps are numerous. The simplest one is the
pairing $\mathbb{R\times}E\rightarrow E$, $(\alpha,x)\rightarrow\alpha x,$
associated to any ordered vector space $E.$

If $E$ is a Banach lattice, then the duality bilinear map $B:E\times
E^{\prime}\rightarrow\mathbb{R},$ given by $B(x,x^{\prime})=x^{\prime}(x)$ is
also positive.

When $E,$ $F$ and $G$ are three Banach lattices all isomorphic with $L^{1}%
(\mu)$ spaces or with $L^{\infty}(\mu)$ spaces, then the composition map
$\Phi:L(E,F)\times L(F,G)\rightarrow L(E,G),$ $\Phi(S,T)=T\circ S,$ is a
positive bilinear map. See Schaefer \cite{S1974}, Theorem 1.5, p. 232.

The operator of convolution $\left(  f,g\right)  \rightarrow\int_{\mathbb{R}%
}f(x-y)g(y)dy$ defines a positive bilinear map on $L^{1}(\mathbb{R})\times
L^{1}(\mathbb{R}).$

Last, but not least, the trace functional, defines a positive bilinear map
\[
\Phi:\mathcal{A}(\mathbb{R}^{N})\times\mathcal{A}(\mathbb{R}^{N}%
)\rightarrow\mathbb{R},\quad\Phi(A,B)=\operatorname*{Trace}\left(  AB\right)
.
\]
Indeed, if $A$ and $B$ are positive, then $A^{1/2}BA^{1/2}$ is also a positive
operator and $\operatorname*{Trace}\left(  AB\right)  =\operatorname*{Trace}%
\left(  A^{1/2}BA^{1/2}\right)  $. Notice that $\Phi$ defines a scalar product
on $\mathcal{A}(\mathbb{R}^{N})$ whose associated norm is the Frobenius norm,%
\[
|||A|||=\left(  \operatorname*{Trace}\left(  A^{2}\right)  \right)  ^{1/2}.
\]
This norm is equivalent to the usual operator norm on $\mathcal{A}%
(\mathbb{R}^{N}),$%
\[
\left\Vert A\right\Vert =\sup_{\left\Vert x\right\Vert =1}\left\vert \langle
Ax,x\rangle\right\vert .
\]

\section{Application to the convergence of positive series}

Many tests of convergence for positive\ series extend to the framework of
ordered Banach spaces as sketched in the preceding section. For example, so is
the case of Olivier's test of convergence:

\begin{theorem}
\label{thmolv}Suppose that $\Phi:E\times F\rightarrow G$ is a positive
bilinear map acting on ordered Banach spaces and $(x_{n})_{n}$ and
$(y_{n})_{n}$ are two sequences of positive elements belonging respectively to
$E$ and $F$ that fulfil the following conditions:

$(a)$ $(x_{n})_{n}$ is decreasing and $\left\Vert x_{n}\right\Vert
\rightarrow0;$

$(b)$ The series $\sum\Phi(x_{n},y_{n})$ is convergent.

Then%
\[
\lim_{n\rightarrow\infty}\Phi\left(  x_{n},\sum_{k=1}^{n}y_{k}\right)  =0.
\]

\end{theorem}

\begin{proof}
Indeed, for $\varepsilon>0$ arbitrarily fixed one can find an index $N>1$ such
that $\left\Vert \sum_{k=N}^{\infty}\Phi(x_{k},y_{k})\right\Vert
<\varepsilon/2.$ Then the inequalities
\[
0\leq\Phi(x_{n},\sum_{k=N}^{n}y_{k})\leq\sum_{k=N}^{n}\Phi(x_{k},y_{k}%
)\leq\sum_{k=N}^{\infty}\Phi(x_{k},y_{k}),
\]
yield $\left\Vert \Phi(x_{n},\sum_{k=N}^{n}y_{k})\right\Vert <\varepsilon/2$
for every $n\geq N.$\ Since $\left\Vert x_{n}\right\Vert \rightarrow0$ and%
\[
\left\Vert \Phi(x_{n},y_{k})\right\Vert \leq\left\Vert \Phi\right\Vert
\left\Vert x_{n}\right\Vert \sup_{1\leq k\leq N-1}\left\Vert y_{k}\right\Vert
\]
for every $k=1,...,N-1,$ we infer the existence of an index $\tilde{N}$ such
that for every $n\geq\tilde{N},$
\[
\sup_{1\leq k\leq N-1}\left\Vert \Phi(x_{n},y_{k})\right\Vert <\varepsilon
/2N.
\]
\newline Therefore%
\begin{align*}
\left\Vert \Phi\left(  x_{n},\sum_{k=1}^{n}y_{k}\right)  \right\Vert  &
\leq\sum_{k=1}^{N-1}\left\Vert \Phi(x_{n},y_{k})\right\Vert +\left\Vert
\Phi(x_{n},\sum_{k=N}^{n}y_{k})\right\Vert \\
&  <\varepsilon/2+\varepsilon/2=\varepsilon
\end{align*}
for every $n\geq\tilde{N}$ and the proof is done.
\end{proof}

\begin{corollary}
\label{corOlv}If $\sum x_{n}$ is a convergent series of positive elements in
an ordered Banach space $E$ and the sequence $(x_{n})_{n}$ is decreasing, then
$n\left\Vert x_{n}\right\Vert \rightarrow0.$
\end{corollary}

Olivier's test of convergence represents the scalar case of Corollary
\ref{corOlv}. In his paper from 1827, Olivier wrongly claimed that
$nx_{n}\rightarrow0$ is also a sufficient condition for the convergence of a
numerical positive series whose terms form a sequence decreasing to 0. One
year later, Abel \cite{Ab1828} disproved this claim by considering the case of
the divergent series $\sum\frac{1}{n\log n}.$ See \cite{NPr2014}, for more
details about this story that played an important role in rigorizing the
theory of numerical series.

Theorem \ref{thmolv} allows us to derive an analogue of Abel's partial
summation for series:

\begin{theorem}
Suppose that $\Phi:E\times F\rightarrow G$ is a positive bilinear map acting
on ordered Banach spaces and $(x_{n})_{n}$ and $(y_{n})_{n}$ are two sequences
of positive elements belonging respectively to $E$ and $F$ such that
$(x_{n})_{n}$ is decreasing and $\left\Vert x_{n}\right\Vert \rightarrow0.$
Then the series $\sum\Phi(x_{n},y_{n})$ and $\sum\Phi\left(  x_{n}%
-x_{n+1},\sum_{k=1}^{n}y_{k}\right)  $ have the same nature and in case of
convergence they have the same sum,%
\[
\sum_{n=1}^{\infty}\Phi(x_{n},y_{n})=\sum_{n=1}^{\infty}\Phi\left(
x_{n}-x_{n+1},\sum_{k=1}^{n}y_{k}\right)  .
\]

\end{theorem}

\begin{proof}
One implication follows easily from Theorem \ref{thmolv} and Abel's partial
summation formula (\ref{B1}),%
\[
\sum_{k=1}^{n}B\left(  x_{k},y_{k}\right)  =\sum_{k=1}^{n-1}B\left(
x_{k}-x_{k+1},\sum_{j=1}^{k}y_{j}\right)  +B\left(  x_{n},\sum_{j=1}^{n}%
y_{j}\right)  .
\]

Conversely, suppose the series $\sum_{n=1}^{\infty}\Phi\left(  x_{n}%
-x_{n+1},\sum_{k=1}^{n}y_{k}\right)  $ is convergent. Then, according to our
hypotheses,%
\[
0\leq\Phi\left(  x_{n},\sum_{k=1}^{n}y_{k}\right)  \leq\sum_{k=n}^{\infty}%
\Phi\left(  x_{k}-x_{k+1},\sum_{j=1}^{n}y_{j}\right)  \leq\sum_{k=n}^{\infty
}\Phi\left(  x_{k}-x_{k+1},\sum_{j=1}^{k}y_{j}\right)
\]
and the squeeze theorem allows us to conclude that $\Phi\left(  x_{n}%
,\sum_{k=1}^{n}y_{k}\right)  \rightarrow0.$ The proof ends with a new appeal
to formula (\ref{B1}).
\end{proof}

\begin{corollary}
\label{corabel}Suppose that $\sum x_{n}$ is a convergent series of positive
elements in an ordered Banach space $E.$ Then the series $\sum_{n}\left(
\sum_{k=n}^{\infty}x_{k}\right)  $ and $\sum nx_{n}$ have the same nature and
in the case of convergence they have the same sum,%
\[
\sum_{n=1}^{\infty}\left(  \sum_{k=n}^{\infty}x_{k}\right)  =\sum
_{n=1}^{\infty}nx_{n}.
\]

\end{corollary}

Coming back to Olivier's test of convergence, it is worth noticing that in the
absence of monotonicity, only a weaker form of Corollary \ref{corOlv} holds true.

\begin{lemma}
\label{lemces}If $\sum x_{n}$ is a convergent series of positive elements in
an ordered Banach space $E$, then%
\[
\lim_{n\rightarrow\infty}\frac{1}{n}\sum_{k=1}^{n}kx_{k}=0.
\]

\end{lemma}

\begin{proof}
Indeed, by\ denoting $S_{n}=\sum_{k=1}^{n}x_{k}$ for $n=1,2,3,...,$ the
sequence $(S_{n})_{n}$ is convergent, say to $S.$ According to Ces\`{a}ro's
theorem,%
\[
\lim_{n\rightarrow\infty}\frac{S_{1}+\cdots+S_{n-1}}{n}=S,
\]
whence%
\[
\lim_{n\rightarrow\infty}\frac{a_{1}+2a_{2}+\cdots+na_{n}}{n}=\lim
_{n\rightarrow\infty}\left(  S_{n}-\frac{S_{1}+\cdots+S_{n-1}}{n}\right)  =0.
\]

\end{proof}

If $\sum x_{n}$ is a convergent series of positive elements in a Banach
lattice $E,$ then for every choice of the signs $\pm$ the series $\sum\pm
x_{n}$ is also convergent. Therefore, for every continuous linear functional
$x^{\prime}\in E^{\prime}$ we have%
\[
\lim_{n\rightarrow\infty}\frac{1}{n}\sum_{k=1}^{n}k\left\vert x^{\prime
}\left(  x_{k}\right)  \right\vert =0,
\]
that is, the sequence $(nx_{n})_{n}$ is weakly mixing to 0. See Zsid\'{o}
\cite{Zs2006} for a theory of these sequences.

Suppose now that $E$ is an ordered Banach space with a strong order unit $u>0$
and the norm of $E$ is associated to the strong order unit. This means that%
\[
E=\bigcup\limits_{n=1}^{\infty}[-nu,nu]
\]
and%
\[
\left\Vert x\right\Vert =\inf\left\{  \lambda>0:x\in\lbrack-\lambda u,\lambda
u]\right\}  .
\]
Examples of such spaces are $C(K)$, $L^{\infty}(\mu),$ $c,$ $\ell^{\infty},$
$\mathcal{A}(H)$ etc. For them one can reformulate the conclusion of Lemma
\ref{lemces} in terms of convergence in density.

\begin{definition}
A sequence $(x_{n})_{n}$ of elements belonging to a Banach space $E$
\emph{converges in density} to $x\in E$ $($abbreviated, $(d)$-$\lim
\limits_{n\rightarrow\infty}x_{n}=x)$ if for every $\varepsilon>0$ the set
$A(\varepsilon)=\left\{  n:\left\Vert x_{n}-x\right\Vert \geq\varepsilon
\right\}  $ has zero density, that is,%
\[
\underset{n\rightarrow\infty}{\lim}\frac{\left\vert A(\varepsilon
)\cap\{1,\ldots,n\}\right\vert }{n}=0.
\]

\end{definition}

Here $\left\vert \cdot\right\vert $ stands for cardinality.

Introduced by Koopman and von Neumann in \cite{KN1932}, this concept proved
useful in ergodic theory and its applications. See the monograph of
Furstenberg \cite{F1981}.

The following result provides a discrete analogue of Koopman-von Neumann's
characterization of convergence in density within the framework of ordered
Banach spaces.

\begin{theorem}
\label{ThmKvN}Suppose that $E$ is an ordered Banach space whose norm is
associated to a strong order unit $u>0$. Then for every sequence $(x_{n})_{n}$
of positive elements of $E,$%
\[
\lim_{n\rightarrow\infty}\frac{1}{n}\sum_{k=1}^{n}x_{k}=0\Rightarrow
(d)\text{-}\lim_{n\rightarrow\infty}x_{n}=0.
\]
The converse works under additional hypotheses, for example, for bounded sequences.
\end{theorem}

\begin{proof}
Assuming $\lim\limits_{n\rightarrow\infty}\frac{1}{n}\sum_{k=1}^{n}x_{k}=0,$
we associate to each $\varepsilon>0$ the set $A_{\varepsilon}=\left\{
n\in\mathbb{N}:x_{n}\geq\varepsilon u\right\}  .$ Since%
\begin{align*}
0  &  \leq\frac{\left\vert \{1,...,n\}\cap A_{\varepsilon}\right\vert }%
{n}u\leq\frac{1}{n}\sum_{k=1}^{n}\frac{x_{k}}{\varepsilon}\\
&  \leq\frac{1}{\varepsilon n}\sum_{k=1}^{n}x_{k}\rightarrow0\text{\ as
}n\rightarrow\infty,
\end{align*}
we infer that each of the sets $A_{\varepsilon}$ has zero density. Therefore
$(d)$-$\lim_{n\rightarrow\infty}x_{n}=0.$

Suppose now that $(x_{n})_{n}$ is a bounded sequence and $(d)$-$\lim
\limits_{n\rightarrow\infty}x_{n}=0.$ Since boundedness in norm is equivalent
to boundedness in order, there is a number $C>0$ such that $x_{n}\leq Cu$ for
all $n.$ Then for every $\varepsilon>0$ there is a set $J$ of zero density
outside which $x_{n}<\varepsilon u$ and we have
\begin{align*}
\frac{1}{n}\sum_{k=1}^{n}x_{k}  &  =\frac{1}{n}\sum_{k\in\{1,...,n\}\cap
J}x_{k}+\frac{1}{n}\sum_{k\in\{1,...,n\}\backslash J}x_{k}\\
&  \leq\frac{\left\vert \{1,...,n\}\cap J\right\vert }{n}\cdot Cu+\varepsilon
u
\end{align*}
Since $\lim\limits_{n\rightarrow\infty}\frac{\left\vert \{1,...,n\}\cap
J\right\vert }{n}=0$, we conclude that $\lim\limits_{n\rightarrow\infty}%
\frac{1}{n}\sum_{k=1}^{n}x_{k}=0.$
\end{proof}

\begin{corollary}
\label{corOlvgen}If $\sum x_{n}$ is a convergent series of positive elements
in an ordered Banach space $E$ whose norm is associated to a strong order
unit, then
\[
(d)\text{-}\lim_{n\rightarrow\infty}nx_{n}=0.
\]

\end{corollary}

Simple numerical examples show that the conclusion of Corollary
\ref{corOlvgen} cannot be improved.

\section{A connection with Jensen-Steffensen inequality}

From the bilinear form of Abel's partial summation formula (see $(\Phi A1)$
and $(\Phi A3)$ above) we infer the following result that offers instances
where the sum of non necessarily positive elements is yet nonnegative.

\begin{theorem}
\label{corAbel}Suppose that $E,$ $F$ and $G$ are ordered vector spaces and
$\Phi:E\times F\rightarrow G$ is a positive bilinear map. If $x_{1}%
,x_{2},...,x_{n}\in E$ and $y_{1},y_{2},...,y_{n}\in F$ verify one of the
following two conditions
\begin{align*}
(i)\text{ }x_{1}  &  \geq x_{2}\geq\cdots\geq x_{n}\geq0\text{ and }\sum
_{k=1}^{j}y_{k}\geq0\text{ for all }j\in\{1,2,...,n\},\text{ }\\
(ii)\text{ }0  &  \leq x_{1}\leq x_{2}\leq\cdots\leq x_{n}\text{ and }%
\sum_{k=j}^{n}y_{k}\geq0\text{ for all }j\in\{1,2,...,n\},
\end{align*}
then%
\[
\sum_{k=1}^{n}\Phi\left(  x_{k},y_{k}\right)  \geq0.
\]

\end{theorem}

The alert reader will recognize here the framework of another important result
in real analysis, the Steffensen extension of Jensen's inequality:

\begin{theorem}
\label{thmstf}$($\emph{Steffensen \cite{St1919}}$)$ Suppose that
$x_{1},...,x_{n}$ is a monotonic family of points in an interval $[a,b]$ and
$w_{1},...,w_{n}$ are real weights such that
\begin{equation}
\sum_{k\,=\,1}^{n}\,w_{k}=1\quad\text{and}\quad0\leq\sum_{k\,=\,1}^{m}%
\,w_{k}\leq\sum_{k\,=\,1}^{n}\,w_{k}\quad\text{for every }m\in\{1,...,n\}.
\tag{dSt}%
\end{equation}
Then every convex function $f$ defined on $[a,b]$ verifies the inequality%
\begin{equation}
f\left(  \sum_{k=1}^{n}w_{k}x_{k}\right)  \leq\sum_{k=1}^{n}w_{k}f(x_{k}).
\tag{$JSt$}\label{JSt}%
\end{equation}

\end{theorem}

The proof of Theorem \ref{thmstf} can be easily reduced to the case of
continuous convex functions and next (via an approximation argument) to the
case of piecewise linear convex functions. Taking into account the following
result that describes the structure of piecewise linear convex functions, the
proof of Theorem \ref{thmstf} reduces ultimately to the case of absolute value function.

\begin{theorem}
\label{thmpop}\emph{(Hardy, Littlewood and P\'{o}lya \cite{HLP})} Let
$f\colon\lbrack a,b]\rightarrow\mathbb{R}$ be a piecewise linear convex
function. Then $f$ is the sum of an affine function and a linear combination,
with positive coefficients, of translates of the absolute value function. In
other words, $f$ is of the form
\[
f(x)=\alpha x+\beta+\sum_{k=1}^{n}c_{k}|x-x_{k}|
\]
for suitable $\alpha,\beta,x_{1},...,x_{n}\in\mathbb{R}$ and suitable
nonnegative coefficients $c_{1},\dots,c_{n}$.
\end{theorem}

Simple proofs are available in \cite{Pop1965} and \cite{NP2006}, pp. 34-35.

\begin{proof}
(of Theorem \ref{thmstf}) We already noticed that the critical case is that of
the absolute value function. This can be settled as follows. $/$Assuming the
ordering $x_{1}\leq\cdots\leq x_{n}$ (to make a choice), we infer that%
\[
0\leq x_{1}^{+}\leq\cdots\leq x_{n}^{+}%
\]
and
\[
x_{1}^{-}\geq\cdots\geq x_{n}^{-}\geq0
\]
where $z^{+}=\max\left\{  z,0\right\}  $ and $z^{-}=\max\left\{  -z,0\right\}
$ denotes respectively the positive part and the negative part of any element
$z.$ According to Theorem \ref{corAbel} (applied to the bilinear map
$B(w,x)=wx)$ we have
\[
\sum_{k=1}^{n}w_{k}x_{k}^{+}\geq0\text{ and }\sum_{k=1}^{n}w_{k}x_{k}^{-}%
\geq0
\]
equivalently,%
\[
\left\vert \sum_{k=1}^{n}w_{k}x_{k}\right\vert \leq\sum_{k=1}^{n}%
w_{k}\left\vert x_{k}\right\vert .
\]
and the proof is done.
\end{proof}

As was noticed in \cite{NP2006}, Exercise 3, p. 184, Theorem \ref{thmpop} does
not extend to higher dimensions. However, there is a nontrivial class of
convex functions for which Steffensen's inequality still works. Given an order
interval $[u,v]$ of a Banach lattice $E,$ let us denote by $Cv_{0}([u,v],E)$
the closure (in the point-wise convergence topology) of the convex cone
consisting of all functions $f:[u,v]\rightarrow E$ of the form
\[
f(x)=A(x)+\sum_{k=1}^{n}c_{k}|x-x_{k}|
\]
for some affine function $A:E\rightarrow E,$ some elements $x_{1},...,x_{n}%
\in\lbrack u,v]$ and some positive coefficients $c_{1},...,c_{n}.$ The
functions belonging to $Cv_{0}([u,v],E)$ verify the condition of convexity
\[
f\left(  \left(  1-\lambda\right)  x+\lambda y\right)  \leq\left(
1-\lambda\right)  f(x)+\lambda f(y)
\]
for all $x,y\in E$ and $\lambda\in\lbrack0,1]$ (the inequality taking place in
the ordering of $E$). An inspection of the argument of Theorem \ref{thmstf}
easily shows that this result still works for functions belonging to
$Cv_{0}([u,v],E):$

\begin{theorem}
$($\emph{The generalization of} \emph{Jensen-Steffensen Inequality) } Suppose
that $E$ is a Banach lattice, $x_{1},...,x_{n}$ is a monotonic family of
points in an order interval $[u,v]$ of $E$ and $w_{1},...,w_{n}$ is a family
of real weights. Then every function $f$ belonging to $Cv_{0}([u,v],E)$
verifies the inequality%
\[
f\left(  \sum_{k=1}^{n}w_{k}x_{k}\right)  \leq\sum_{k=1}^{n}w_{k}f(x_{k}).
\]

\end{theorem}

\section{A connection with majorization theory}

The theory of majorization provides a unified approach to the analysis of a
number of models in economics, finance, risk management, genetics etc. and is
masterfully exposed in the book of Marshall, Olkin and Arnold \cite{MOA}.

Given a vector $x\in\mathbb{R}^{N}$ of components $x_{1},...,x_{N},$ let
$x^{\downarrow}$ be the vector with the same entries as $x$ but rearranged in
decreasing order,%
\[
x_{1}^{\downarrow}\geq\cdots\geq x_{N}^{\downarrow}.
\]
The vector $x$ is \emph{submajorized} by another vector $y$ (abbreviated,
$x\mathbf{\prec}_{w}y)$ if%
\[
\sum_{i\,=\,1}^{k}\,x_{i}^{\downarrow}\leq\sum_{i\,=\,1}^{k}\,y_{i}%
^{\downarrow}\quad\text{for }k=1,...,N
\]
and \emph{majorized} (abbreviated, $x\mathbf{\prec}y)$ if in addition%
\[
\sum_{i\,=\,1}^{N}\,x_{i}^{\downarrow}=\sum_{i\,=\,1}^{N}\,y_{i}^{\downarrow
}\,.
\]

The following result outlines a connection between Abel's partial summation
formula and the Tomi\'{c}-Weyl inequality of majorization (\cite{NP2006},
Theorem 1.10.4, p.57):

\begin{theorem}
\label{thmmaj1}Suppose that $\Phi:E\times E\rightarrow G$ is a positive
bilinear map and $x_{1},x_{2},...,x_{n}$ is a decreasing sequence of elements
of $E.$ If $u_{1},u_{2},...,u_{n}$ and $v_{1},v_{2},...,v_{n}$ are two
families of elements of $E$ such that
\[
u_{1}\geq u_{2}\geq\cdots\geq u_{n}\geq0\text{ and }\sum_{k=1}^{j}u_{k}%
\leq\sum_{k=1}^{j}v_{k}\text{ for }j\in\{1,2,...,n\},
\]
then%
\[
\sum_{k=1}^{n}\Phi\left(  x_{k},u_{k}\right)  \leq\sum_{k=1}^{n}\Phi\left(
x_{k},v_{k}\right)  \text{ }%
\]
and%
\[
\sum_{k=1}^{n}\Phi\left(  u_{k},x_{k}\right)  \leq\sum_{k=1}^{n}\Phi\left(
v_{k},x_{k}\right)  \text{.}%
\]

\end{theorem}

\begin{proof}
Indeed, according to (\ref{B1}),%
\begin{align*}
\sum_{k=1}^{n}\Phi\left(  x_{k},u_{k}\right)   &  =\sum_{k=1}^{n-1}\Phi\left(
x_{k}-x_{k+1},\sum_{j=1}^{k}u_{j}\right)  +\Phi\left(  x_{n},\sum_{j=1}%
^{n}u_{j}\right) \\
&  \leq\sum_{k=1}^{n-1}\Phi\left(  x_{k}-x_{k+1},\sum_{j=1}^{k}v_{j}\right)
+\Phi\left(  x_{n},\sum_{j=1}^{n}v_{j}\right)  =\sum_{k=1}^{n}\Phi\left(
x_{k},v_{k}\right)  .
\end{align*}

On the other hand from (\ref{B2}) we infer that%
\begin{align*}
\sum_{k=1}^{n}\Phi\left(  u_{k},x_{k}\right)   &  =\sum_{k=1}^{n-1}\Phi\left(
\sum_{j=1}^{k}u_{j},x_{k}-x_{k+1}\right)  +\Phi\left(  \sum_{j=1}^{n}%
u_{j},x_{n}\right) \\
&  \leq\sum_{k=1}^{n-1}\Phi\left(  \sum_{j=1}^{k}v_{j},x_{k}-x_{k+1}\right)
+\Phi\left(  \sum_{j=1}^{n}v_{j},x_{n}\right)  =\sum_{k=1}^{n}\Phi\left(
v_{k},x_{k}\right)  .
\end{align*}

\end{proof}

In the particular case where $E=\mathcal{A}(\mathbb{R}^{N}),$ $G=\mathbb{R}$
and $\Phi(A,B)=\operatorname*{Trace}\left(  AB\right)  $, Theorem
\ref{thmmaj1} yields the inequality%
\[
\sum_{k=1}^{n}\operatorname*{Trace}A_{k}^{2}\leq\sum_{k=1}^{n}%
\operatorname*{Trace}A_{k}B_{k},
\]
provided that the self-adjoint operators $A_{k}$ and $B_{k}$ verify the
conditions%
\[
A_{1}\geq A_{2}\geq\cdots\geq A_{n}\geq0\text{ and }\sum_{k=1}^{j}A_{k}%
\leq\sum_{k=1}^{j}B_{k}\text{ for }j\in\{1,2,...,n\}
\]

Combining this with the Cauchy-Schwarz inequality,%
\[
\left(  \sum_{k=1}^{n}\operatorname*{Trace}A_{k}B_{k}\right)  ^{2}\leq\left(
\sum_{k=1}^{n}\operatorname*{Trace}A_{k}^{2}\right)  \left(  \sum_{k=1}%
^{n}\operatorname*{Trace}B_{k}^{2}\right)  ,
\]
we arrive at the following trace inequality ascribed to K. L. Chung:%
\[
\sum_{k=1}^{n}\operatorname*{Trace}A_{k}^{2}\leq\sum_{k=1}^{n}%
\operatorname*{Trace}B_{k}^{2}.
\]

The function $A\rightarrow\operatorname*{Trace}\left(  f(A)\right)  $ is
convex on $\mathcal{A}(\mathbb{R}^{N})$ whenever $f:\mathbb{R\rightarrow R}$
is a convex function. See \cite{P}, Proposition 2, p. 288. Thus, Chung's
inequality is an illustration of the following trace analogue of
Tomi\'{c}-Weyl inequality of submajorization:

\begin{theorem}
\label{thmmaj2}Let $f:\mathbb{R}\rightarrow\mathbb{R}$ be a nondecreasing
convex function. If $A_{1},A_{2},...,A_{n}$ and $B_{1},B_{2},...,B_{n}$ are
two families of elements of $\mathcal{A}(\mathbb{R}^{N})$ such that
\[
A_{1}\geq A_{2}\geq\cdots\geq A_{n}\geq0\text{ and }\sum_{k=1}^{j}A_{k}%
\leq\sum_{k=1}^{j}B_{k}\text{ for }j\in\{1,2,...,n\},
\]
then%
\[
\sum_{k=1}^{n}\operatorname*{Trace}f\left(  A_{k}\right)  \leq\sum_{k=1}%
^{n}\operatorname*{Trace}f\left(  B_{k}\right)  .
\]

\end{theorem}

\begin{proof}
We will consider here the case where $f$ is continuously differentiable. The
general case can be deduced from this one by using approximation arguments.
Since the function $A\rightarrow\operatorname*{Trace}f(A)$ is convex on
$\mathcal{A}(\mathbb{R}^{N})$, for each $\lambda\in(0,1]$ we have
\[
\frac{\operatorname*{Trace}f\left(  A+\lambda(X-A)\right)
-\operatorname*{Trace}f(A)}{\lambda}\leq\operatorname*{Trace}%
f(X)-\operatorname*{Trace}f(A),
\]
whence we infer (by letting $\lambda\rightarrow0)$ that%
\[
\operatorname*{Trace}\left[  f^{\prime}(A)(X-A)\right]  \leq
\operatorname*{Trace}f(X)-\operatorname*{Trace}f(A).
\]
According to the bilinear form of Abel's partial summation formula
(\ref{B1}),
\begin{multline*}
\sum_{k=1}^{n}\left[  \operatorname*{Trace}f(B_{k})-\operatorname*{Trace}%
f(A_{k})\right]  \geq\sum_{k=1}^{n}\operatorname*{Trace}\left[  f^{\prime
}(A_{k})(B_{k}-A_{k})\right]  \\
=\sum_{k=1}^{n}\operatorname*{Trace}\left[  f^{\prime}(A_{n})\sum_{k=1}%
^{n}\left(  B_{k}-A_{k}\right)  \right]  \\
\qquad+\sum_{m=1}^{n-1}\operatorname*{Trace}\left[  (f^{\prime}(A_{m}%
)-f^{\prime}(A_{m+1}))\sum_{k=1}^{m}(B_{k}-A_{k})\right]
\end{multline*}
and the right hand side is a sum of nonnegative terms due to the fact that
\[
U\leq V\text{ in }\mathcal{A}(\mathbb{R}^{N})\text{ implies }%
\operatorname*{Trace}h(U)\leq\operatorname*{Trace}h(V).
\]
for all increasing and continuous functions $h:\mathbb{R\rightarrow R}.$ See
\cite{P}, Proposition 1, p. 288. The proof ends by noticing that the
derivative of any continuously differentiable function is increasing and continuous.
\end{proof}

An inspection of the argument of Theorem \ref{thmmaj2} shows that this result
also works for nondecreasing convex functions $f$ defined on an arbitrary
interval $I$ provided that they are continuous and the spectra of operators
$A_{k}$ and $B_{k}$ are included in $I.$ The variant of Theorem \ref{thmmaj2}
for log convex functions (such as $\operatorname*{Trace}\left(  e^{A}\right)
$) can be easily obtained using the same idea.

\medskip

\textbf{Acknowledgements}

The authors would like to thank Flavia-Corina Mitroi-Symeonidis for her
valuable comments.

The research of the first author was supported by Romanian National Authority
for Scientific Research CNCS -- UEFISCDI grant PN-II-ID-PCE-2011-3-0257.

\end{document}